\documentclass[10pt,a4paper]{amsart}
\usepackage[english]{babel}
\usepackage{amssymb,xypic,amscd,amsthm,stmaryrd, mathtools, amsmath}

\CompileMatrices
\usepackage[T1]{fontenc}
\newtheorem{defn}{Definition}
\newtheorem{thm}[defn]{Theorem}
\newtheorem{cor}[defn]{Corollary}
\newtheorem{lem}[defn]{Lemma}
\newtheorem{prop}[defn]{Proposition}

\newtheorem{cex}[defn]{Counterexample}
\theoremstyle{plain}
\newtheorem{rem}[defn]{Remark}
\theoremstyle{remark}

\numberwithin{equation}{section} \numberwithin{defn}{section}

\newcommand{\Mat}{\operatorname{Mat}}
\newcommand\ed{\operatorname{End}}
\newcommand{\Img}{\operatorname{Im}}

\newcommand\Ker{\operatorname{Ker}}
\newcommand\aut{\operatorname{Aut}}

\newcommand{\f}{\varphi}
\newcommand{\g}{\psi}
\newcommand{\R}{\mathbb{R}}
\newcommand{\N}{\mathbb{N}}

\newcommand{\Z}{\mathbb{Z}}
\newcommand{\C}{\mathbb{C}}
\newcommand{\h}{\mathcal{H}}
\newcommand\Id{\operatorname{Id}}
\begin{document}

\title[Core Partial Order]{Core Partial Order for Finite Potent endomorphisms}
\author{ Diego Alba Alonso* }

\address{ Departamento de Matem\'aticas, Universidad de Salamanca, Plaza de la Merced 1-4, 37008 Salamanca, Espa\~na}
 \email{ (*) daa29@usal.es, ORCID:0000-0001-7147-8368}

\thanks{Partially supported by the Spanish Government research project PID2023-150787NB-100.}

\begin{abstract}
The aim of this paper is to generalize the Core Inverse to arbitrary vector spaces using finite potent endomorphisms. As an application, the core partial order is studied in the set of finite potent endomorphisms (of index lesser or equal than one), thus generalizing the theory of this order to infinite dimensional vector spaces. Moreover, a pre-order is presented using the CN-decomposition of a finite potent endomorphism. Finally, some questions concerning this pre-order are posed. Throughout the paper, some remarks are also made in the framework of arbitrary Hilbert spaces using bounded finite potent endomorphisms.
\end{abstract}

\maketitle

\bigskip

\setcounter{tocdepth}1

\tableofcontents
\bigskip

\medskip

\textbf{Mathematical Subject Classification}: 06A06, 15A03, 15A04, 15A09, 46C99.\\ \bigskip 

\textbf{Keywords}: Finite Potent Endomorphisms, Hilbert Spaces , Bounded Finite Potent Endomorphisms, Core Inverse, Core Partial Order.
\section{Introduction.}
In this paper, the set of all $m\times n$ matrices over a field $k$ is represented by $\Mat_{m\times n}(k).$ Let $A\in \Mat_{m\times n}(k),$ the symbols $A^{-1}, A^*, \mathcal{R}(A), \mathcal{N}(A),\mathrm{rk(A)}$ and $P_{\mathcal{R}(A)}$ will stand for the usual inverse (when it exists), the conjugate transpose, column space, null space, the rank of matrix $A$ and the orthogonal projection onto the column space of $A$ respectively. Moreover, $\mathrm{Id}\in \Mat_{n\times n}(k)$ will denote the identity matrix.\\

For an arbitrary $(n\times n)-$matrix $A$ with entries in a field $k$, the index of $A,$ $i(A)\geq 0,$ is the smallest non-negative integer such that $\mathrm{rk(A^{i(A)})}=\mathrm{rk(A^{i(A)+1})}.$ The Drazin inverse of $A$ is the unique solution that satisfies: $$A^{i(A)+1}\cdot X=A^{i(A)}\quad \& \quad X\cdot A \cdot X=X\quad \& \quad A\cdot X=X\cdot A, $$ and it is represented by $A^D.$ If $i(A)\leq 1,$ then $A^D$ is called the group inverse of $A$ and it is denoted by $A^{\#}.$ \\ 
For $A\in \Mat_{m\times n}(\C),$ if $X\in \Mat_{n\times m}(\C)$ satisfies $$A\cdot X \cdot A=A \quad\& \quad X\cdot A \cdot X= X\quad \& \quad (A\cdot X)^*=X\cdot A\quad \&\quad (X\cdot A)^*=X\cdot A, $$ then $X$ is called Moore-Penrose inverse of $A.$ This matrix is unique and it is represented by $A^{\dagger}.$ For details see \cite{Ben}.

In 2010, Baksalary and Trenkler introduced in \cite{BakTren} the notion of the Core Inverse of a matrix $A\in \Mat_{n\times n}(\C)$ with $i(A)\leq 1$ as the only matrix $X\in \Mat_{n\times n}(\C)$ satisfying the following conditions: $$A\cdot X= P_{\mathcal{R}(A)} \text{ and } \mathcal{R}(X)\subseteq \mathcal{R}(A). $$ It is denoted as $A^{\textcircled \#}.$ Initially, it was presented as an alternative to the group inverse.
  Firstly, they noted that this generalized inverse could be used to solve certain linear systems and gained interest as it coincides with Bott-Duffin inverse $P_{\mathcal{R}(A)}\cdot[(A-\Id)\cdot P_{\mathcal{R}(A)}+\Id]^{-1}.$ Bott-Duffin inverse occurs in the solutions of some constrained systems of equations arising in Electrical Network Theory, see \cite[Chapter 2, Section 10]{Ben}, \cite{Camp}. 
Ever since its appearance, the core inverse and its applications have interested a lot of researchers in the area. See for instance: \cite{Di1}, \cite{Di2}, \cite{Huan}, \cite{Ke}, \cite{Kurata}, \cite{Malik}.\\ 
  John Tate, in \cite{Ta}, introduced the notion of finite potent endomorphism. Let $k$ be an arbitrary field and let $V$ be an arbitrary $k-$vector space. Let us now consider an endomorphism $\f$ of $V.$ We say that $\f$ is finite potent if $\f^n(V)$ is a finite dimensional vector subspace for some $n.$ Tate used this operators in order to give an intrinsic definition of Abstract Residue.\\
In \cite{BakTren}, the authors studied some properties of the core inverse using singular value decomposition and they used this inverse to define a partial order in the class of square complex matrices of index one. In this paper, the core inverse and the core order are extended to arbitrary vector spaces, in general, infinite dimensional ones, endowed with an inner product over a field $k=\R$ or $k=\C ,$ using finite potent endomorphisms.\smallskip \\ 

The article is organized as follows. \\Firstly, Section \ref{s:pre} is a gathering of results that will be used later. Section \ref{S: MP bounded operator} contains some results related to the Moore-Penrose inverse of a bounded finite potent endomorphism. Briefly, given a finite potent endomorphism $\f \in \ed_k(V),$ the study of the Moore-Penrose inverse within this framework was done by asking for an admissibility condition. To wit, the Moore-Penrose of a finite potent endomorphism exists when the vector space admits the following decompositions: $$\f \colon V=\Ker(\f)\oplus [\Ker(\f)]^{\perp} \to V=\Img(\f)\oplus [\Img(\f)]^{\perp}.$$ However, being a bounded finite potent operator, could somehow imply any equivalent condition to this one. We devote the first part of this section to answer this question. Moreover, we include a detailed study of the Moore-Penrose inverse of a bounded finite potent operator of index lesser or equal to one. Section \ref{s:core-inv-fp} includes the generalization of the theory related to the Core Inverse to arbitrary vector spaces, following \cite{BakTren} as a guideline. However, the study here presented about the Core Inverse is done with a slightly different approach to the case of matrices. In particular, we do not start by requiring our finite potent endomorphism to be of index lesser or equal than one. We start by deducing that when the Core Inverse of our finite potent endomorphism exists, then the endomorphism shall be of index lesser or equal than one (Proposition \ref{P: ConsecuenciasDefinicion}). From this previous fact and the definition of the Core Inverse we deduce that the Core Inverse is also a unique finite potent endomorphism of index lesser or equal than one (Corollary \ref{C: fcircledfpdeind1}). Further, the well known algebraic characterization of the Core Inverse, involving the group inverse and the Moore-Penrose inverse is presented, as well as a geometric characterization. Using the obtained geometric characterization, we further calculate the Moore-Penrose inverse and the Core Inverse of the Core Inverse.\\ As an application of the previous theory, the core partial order is also generalized to arbitrary vector spaces, namely; infinite dimensional ones, using finite potent endomorphisms and this is included in Section \ref{s: Core Partial Order}. In the same paper as the Core Inverse was introduced, \cite{BakTren}, the core partial order for complex matrices was extensively investigated. After relating the core order for finite potent endomorphisms with the space pre-order, a characterization is derived in Theorem \ref{T: CaractCore}. This is the key element to prove that the core order is indeed a partial order in the set of finite potent endomorphisms of index lesser or equal than one. Finally, in Section \ref{S: PreOrder} we introduce the so-called ``general core order'', which is a pre-order in the set of finite potent endomorphisms of arbitrary index. Moreover, some conjectures are posed concerning this pre-order, possibly relating it with the Core-EP pre-order and, in general, relating the theory exposed here with the Core-EP Inverse.

 It is worth noticing two last things. Every proof and result presented can be specialized to finite square matrices over arbitrary ground fields. Finite potent endomorphisms do not form an ideal of the endomorphisms. Namely, the sum and the composition of two finite potent endomorphisms is not, in general, a finite potent endomorphism. Therefore, the generalization here presented is not merely a generalization from finite dimensional vector spaces (finite square matrices) to infinite dimensional vector spaces, but it also carries the additional problems derived from the impossibility to use the usual ring structure endomorphisms have to generalize the theory.

\medskip

\section{Preliminaries} \label{s:pre}
This section is included for the sake of completeness.

\subsection{Finite Potent Endomorphisms} \label{ss:finite-potent}

Let $k$ be an arbitrary field and let $V$ be a $k$-vector space. Let us now consider an endomorphism $\varphi$ of $V$. We say
that $\varphi$ is ``finite potent'' if $\varphi^n V$ is finite
dimensional for some $n$. This definition was introduced by J. Tate in \cite{Ta} as a basic tool for his elegant definition
of Abstract Residues. 

 In 2007, M. Argerami, F. Szechtman and R. Tifenbach showed in \cite{AST} that an endomorphism $\varphi$ is
finite potent if and only if $V$ admits a $\varphi$-invariant
decomposition $V = U_\varphi \oplus W_\varphi$ such that
$\varphi_{\vert_{U_\varphi}}$ is nilpotent, $W_\varphi$ is finite
dimensional and $\varphi_{\vert_{W_\varphi}} \colon W_\varphi
\overset \sim \longrightarrow W_\varphi$ is an isomorphism.
Hence, this decomposition is unique. We shall call this decomposition the $\varphi$-invariant AST-decomposition of $V$.

Moreover,  we shall call ``index of $\varphi$'', $i(\varphi)$, to the nilpotency order of $\varphi_{\vert_{U_\varphi}}$.  One has that $i(\varphi) = 0$ if and only if $V$ is a finite-dimensional vector space and $\varphi$ is an automorphism.

We shall remark that the sum and the composition of finite potent endomorphism is not necessarily a finite potent endomorphism as can be seen by the following example. Let us consider the $k-$vector space  $V=\underset{i\in \N}{\oplus}<v_i>.$ Moreover, let us define the following endomorphisms:
$$\f(v_{i}) = \left \{ \begin{array}{ccl}  v_{i+1} & \text{ if } &  \text{ i is odd } \\  0 & \text{ if } & \text{ i is even } \end{array} \right . \text{ and }\quad \g(v_{i}) = \left \{ \begin{array}{ccl} 0 & \text{ if } &  \text{ i is odd }  \\   v_{i-1} & \text{ if } & \text{ i is even } \end{array} \right .$$
Notice that both of them are finite potent endomorphisms, as they are nilpotent. Then:
$$(\f+\g)(v_{i}) = \left \{ \begin{array}{ccl}  v_{i+1} & \text{if} & \text{ i is odd } \\  v_{i-1} & \text{ if } & \text{ i is even } \end{array} \right . \text {and }\quad (\f\circ \g)(v_{i}) = \left \{ \begin{array}{ccl}  0 & \text{if} & \text{ i is odd }  \\  v_i & \text{if} & \text{ i is even }  \end{array} \right .,$$ from where we deduce that the sum and the composition of finite potent endomorphisms is not a finite potent endomorphism.

Basic examples of finite potent endomorphisms are all endomorphisms of a finite-dimensional vector space and finite rank or nilpotent endomorphisms of infinite-dimensional vector spaces.

For more details on the theory of finite potent endomorphisms, the reader is referred to \cite{Pa}, \cite{Pa-CN} and \cite{Fpa-CN}.

\medskip

\subsubsection{CN Decomposition of a Finite Potent Endomorphism}\label{ss: CNdecomp01010101}

Let $V$ be again an arbitrary $k$-vector space. Given a finite potent endomorphism $\varphi \in \ed_k (V)$,  there exists a unique decomposition $\varphi = \varphi_{_1} + \varphi_{_2}$, where $\varphi_{_1}, \varphi_{_2} \in \ed_k (V)$ are finite potent endomorphisms satisfying that:

\begin{itemize}

\item $i(\varphi_{_1}) \leq 1$;

\item $\varphi_{_2}$ is nilpotent;

\item $\varphi_{_1} \circ \varphi_{_2} = \varphi_{_2} \circ \varphi_{_1} = 0$.

\end{itemize}

Also, the following hold:  \begin{equation} \label{eq:index1} \varphi = \varphi_1 \Longleftrightarrow U_\varphi = \Ker \varphi \Longleftrightarrow  W_\varphi = \text{ Im } \varphi  \Longleftrightarrow i(\varphi) \leq 1\, .\end{equation}

Moreover, if $V = W_{_\varphi}\oplus U_{_\varphi}$ is the AST-decomposition of $V$ induced by $\varphi$, then $\varphi_{_1}$ and $\varphi_{_2}$ are the unique linear maps such that:

\begin{equation} \label{eq:expl-CN-exp-3498353} \varphi_{_1} (v) = \left \{ \begin{aligned} \varphi (v) \, &\text{ if } \, v\in W_{_\varphi} \\ \, 0 \quad &\text{ if } \, v\in U_{_\varphi} \end{aligned} \right . \quad \text{ and } \quad \varphi_{_2} (v) = \left \{ \begin{aligned} \, 0 \quad &\text{ if } \, v\in W_{_\varphi} \\ \varphi (v) \, &\text{ if } \, v\in U_{_\varphi} \end{aligned} \right . \quad \, .\end{equation}

\medskip

\subsection{Bounded finite potent endomorphisms on Hilbert spaces.}\label{ss: BFP}

In 2021, the author of \cite{Fpa-Boun} studied the set of bounded finite potent endomorphisms on arbitrary Hilbert spaces. Henceforth, this set will be denoted as $B_{fp}(\h).$ \begin{thm}\label{T: Charact BFP}\cite[Theorem 3.7]{Fpa-Boun}(Characterization of bounded finite potent endomorphisms). Given a Hilbert space $\h$ and an endomorphism $\f \in \ed_{\C}(\h),$ then the following conditions are equivalent: \begin{itemize}
\item $\f \in B_{fp}(\h);$ 
\item $\h$ admits a decomposition $\h=W_{\f}\oplus U_{\f}$ where $W_{\f}$ and $U_{\f}$ are closed $\f-$invariant subspaces of $\h,$ $W_{\f}$ is finite- dimensional, $\f_{\vert W_{\f}}$ is an homeomorphism of $W_{\f}$ and $\f_{\vert U_{\f}}$ is a bounded nilpotent operator.
\item $\f$ has a decomposition $\f=\g + \phi,$ where $\g$ is a bounded finite rank operator, $\phi$ is a bounded nilpotent operator and $\g \circ \phi =0=\phi \circ \g.$
\end{itemize} 
\end{thm}

\subsubsection{The adjoint operator of a bounded finite potent endomorphism.}\label{ss: AdjoiintBFP}

Let us now consider two inner product vector spaces $(V,g)$ and $(H, \bar{g}).$ If $\f \colon V \to H$ is a linear map, a linear operator $\f^* \colon H \to V$ is called the \textbf{adjoint} of $\f$ when $$g(\f^*(h),v)=\bar{g}(h, \f(v)),$$ for all $v\in V$ and $h\in H.$ If $\f \in \ed_k(V),$ we say that $\f$ is \textbf{self-adjoint} when $\f = \f^*.$ The existence and uniqueness of the adjoint $\f^*$ of a bounded (or equivalently a continous) operator on arbitrary Hilbert spaces is immediately deduced form the Riesz Representation Theorem. Moreover, the adjoint of a bounded linear map is also bounded. In \cite[Section 4]{Fpa-Boun} the author studied the structure of the adjoint of a bounded finite potent endomorphism. Let us recall some of the results presented in there. If $\f \in B_{fp}(\h),$ with $\h=W_{\f}\oplus U_{\f}$ is the AST-decomposition induced by $\f$ and $\f=\f_1+\f_2$ is the CN-decomposition, then the adjoint operator $\f^*$ has the following properties: \begin{itemize}
\item[I.)]$\f^*\in B_{fp}(\h);$
\item[II.)]$i(\f^*)=i(\f);$
\item[III.)]$\f^*=(\f_1)^*+(\f_2)^*$ is the CN-decomposition of $\f^*;$
\item[IV.)]If $\h=W_{\f^{*}}\oplus U_{\f^{*}}$ is the AST-decomposition induced by $\f^*$ (notice this has sense due to $I$), then one has that $W_{\f^*}=[U_{\f}]^{\perp}$ and $U_{\f^*}=[W_{\f}]^{\perp}.$
\end{itemize}

\subsection{Generalized Inverses.}\label{ss: GenInv.}
 If $A\in \Mat_{n\times m}(k)$ is a matrix with entries in an arbitrary field $k,$ a matrix $A^{-}\in \Mat_{m\times n}(k)$ is a $\{1\}-$inverse of $A$ when $AA^{-}A=A$ and it is a $\{2\}-$inverse of $A$ when $A^{-}AA^{-}=A^-.$ Moreover, we say that a matrix $A^{+}\in \Mat_{m\times n}(k)$ is a reflexive generalized inverse of $A$ when $A^{+}$ is a $\{1\}-$inverse of $A$ and $A$ is a $\{1\}-$inverse of $A^{+},$ this is, $AA^{+}A=A$ and $A^{+}AA^{+}=A^{+}.$ Similarly, given two $k-$vector spaces $V$ and $W$ and a linear map $\f\colon V \to W,$ we will say thay a morphism $\f^{-}\colon W \to V$ is a $\{1\}-$inverse of $\f$ when $\f \circ \f^{-} \circ \f=\f$ and it is a $\{2\}-$inverse of $\f$ when $\f^-\circ\f\circ\f^-=\f^-.$ Similarly, a linear map $\f$  a linear map $\f^{+}\colon W \to V$ is a reflexive generalized inverse of $\f$ when $\f^{+}$ is a $\{1\}-$inverse of $\f$ and $\f$ is a $\{1\}-$inverse of $\f^{+}.$ Given any lineal operator $\f$ then we will denote as: $X_{\f}(1), X_{\f}(2), X_{\f}(1,2)$ the sets of $\{1\}-$inverses, $\{2\}-$inverses and the set of reflexive generalized inverses of $\f$  respectively.

\subsection{Group Inverse of Finite Potent Endomorphisms}  \label{ss:grop-matrix-fp}

Let $V$  be an arbitrary $k$-vector space and let $\varphi \in \ed_k (V)$ be a finite potent endomorphism of $V$. We say that a linear map $\varphi^{\#} \in \ed_k (V)$ is a group inverse of $\varphi$ when it satisfies the following properties:

\begin{itemize}
\item $\varphi \circ \varphi^{\#} \circ \varphi = \varphi$;

\item $\varphi^{\#} \circ \varphi \circ \varphi^{\#} = \varphi^{\#}$;

\item $\varphi^{\#} \circ \varphi = \varphi \circ \varphi^{\#}$.

\end{itemize}

 According to \cite[Lemma 3.4]{FpaGroup} we know that if there exists a group inverse $\varphi^{\#} \in \ed_k (V)$, then $i(\varphi) \leq 1$. Moreover \cite[Theorem 3.5]{FpaGroup} shows that $\varphi^D = \varphi^{\#}$ is the unique group inverse of $\varphi$, where $\varphi^D$ is its Drazin inverse.

The group inverse $\varphi^{\#}$ satisfies the following properties:

\begin{itemize}

\item $(\varphi^{\#})^{\#} = \varphi$;

\item $\varphi = \varphi^{\#}$ if and only if $(\varphi_{\vert_{W_\varphi}})^2 = \text{Id}_{\vert_{W_\varphi}}$;

\item if $n\in {\mathbb Z}^+$, then $(\varphi^n)^{\#} = (\varphi^{\#})^n$.

\end{itemize}

\medskip

\subsection{Moore-Penrose inverse of a bounded linear map.}\label{ss: Moore-Penrose}

 Let $(V,g)$ and $(H,\bar{g})$ be inner product spaces over $k,$ with $k=\C$ or $k=\R.$ 
\begin{defn}\label{D: Admissible}
Given a linear map $\f\colon V \to H,$ we say that $\f$ is admissible for the Moore-Penrose inverse when $V=\Ker(\f) \oplus [\Ker(\f)]^{\perp}$ and $H=\Img(\f) \oplus [\Img(\f)]^{\perp}.$
\end{defn}
According to \cite[Theorem 3.12]{MPFP}, if $(V,g)$ and $(H,\bar{g})$ are inner product spaces over $k,$ then $\f \colon V \to H$ is a linear map admissible for the Moore-Penrose inverse if and only if there exists a unique linear map $\f^{\dagger}\colon H \to V$ such that: \begin{itemize}
\item[I.)] $\f^{\dagger}$ is a reflexive generalized inverse of $\f;$
\item[II.)] $\f^{\dagger}\circ \f$ and $\f \circ \f^{\dagger}$ are self-adjoint, that is:\begin{itemize}
\item[a)] $g([\f^{\dagger}\circ \f](v),v')=g(v,[\f^{\dagger}\circ \f](v'));$
\item[b)] $\bar{g}([\f \circ \f^{\dagger}](h),h')=\bar{g}(h,[\f \circ \f^{\dagger}](h'));$
\end{itemize}
for all $v,v'\in V$ and $h,h'\in H.$ The operator $\f^{\dagger}$ is named the \textbf{Moore-Penrose inverse} of $\f$ and it is the unique linear map satisfying that: $$\f^{\dagger}(h) = \left \{ \begin{array}{ccl} (\f_{\vert [\Ker(\f)]^{\perp}})^{-1}(h) & \text{ if } & h\in \Img(\f) \\  0 & \text{ if } & h\in [\Img(\f)]^{\perp} \end{array} \right . .$$ 
\end{itemize}
The Moore-Penrose inverse $\f^{\dagger}\colon H\to V$ also satisfies the following properties: \begin{itemize}
\item $\f^{\dagger}$ is also admissible for the Moore-Penrose and $(\f^{\dagger})^{\dagger}=\f;$

\item If $\f\in \ed_k(V)$ and $\f$ is an isomorphism, then $\f^{\dagger}=\f^{-1};$

\item $\f^{\dagger}\circ \f=P_{[\Ker(\f)]^{\perp}};$

\item $\f\circ \f^{\dagger}=P_{\Img(\f)};$
\end{itemize}
where $P_{[\Ker(\f)]^{\perp}}$ and $P_{\Img(\f)}$ are the projections induced by the decompositions $V=\Ker(\f) \oplus [\Ker(\f)]^{\perp}$ and $H=\Img(\f) \oplus [\Img(\f)]^{\perp}$ respectively.
\subsubsection{Moore-Penrose inverse of a bounded linear map.}Finally, let us recall some properties of the Moore-Penrose inverse of a bounded linear map between two Hilbert spaces. Let $\h_1$ and $\h_2$ two Hilbert spaces. Given a linear map $\f \colon \h_1 \to \h_2$ that is admissible for Moore-Penrose inverse, this is, $\h_1=\Ker(\f) \oplus [\Ker(\f)]^{\perp}$ and $\h_2=\Img(\f) \oplus [\Img(\f)]^{\perp},$ it is well known that:
\begin{lem}\label{L: Adsii ImgCer}
If $\f \in B(\h_1,\h_2),$ then $\f$ is admissible for the Moore-Penrose inverse if and only if $\Img(\f)$ is a closed subspace of $\h_1.$
\end{lem}
Also, it is well known that: \begin{itemize}
\item If $\f \in B(\h_1,\h_2)$ is admissible for the Moore-Penrose inverse, then $\f^{\dagger}\in B(\h_2, \h_1).$
\item If $\f \in B(\h_1,\h_2)$ is admissible for the Moore-Penrose inverse, then $\f^*$ is also admissible for the Moore-Penrose inverse and $(\f^*)^{\dagger}=(\f^{\dagger})^*.$
\end{itemize}
From the properties of the Moore-Penrose inverse of a linear map, if $\f \in B(\h_1,\h_2)$ with $\Img(\f)$ being a closed subspace of $\h_2,$ one has that: \begin{itemize}
\item $\f^*\circ (\f^*)^{\dagger}=P_{[\Ker(\f)]^{\perp}};$
\item $(\f^*)^{\dagger}\circ \f^*=P_{\Img(\f)}.$
\end{itemize} where $P_{[\Ker(\f)]^{\perp}}$ and $P_{\Img(\f)}$ are the projections induced by the decompositions $\h_1=\Ker(\f) \oplus [\Ker(\f)]^{\perp}$ and $\h_2=\Img(\f) \oplus [\Img(\f)]^{\perp}$ respectively. Moreover, the following usual relations between the adjoint and the Moore-Penrose hold:

\begin{lem}\label{L: Arit MPTRasp}
If $\f \in B(\h_1,\h_2)$ is such that $\Img(\f)$ is a closed subspace of $\h_2,$ then one has that:
\begin{itemize}
\item[I.)] $\f^* \circ \f \circ \f^{\dagger}=\f^*;$
\item[II.)] $\f^{\dagger}\circ \f \circ \f^* = \f^*;$
\item[III.)] $(\f^*)^{\dagger}\circ \f^* \circ \f =\f;$
\item[IV.)] $\f\circ \f^*\circ (\f^*)^{\dagger}=\f.$
\end{itemize}
\end{lem}

\subsection{Space Pre-Order}\label{ss: SpacePreord}
The space pre-order was introduced in \cite[Section 3.2]{Ind} as a tool to study most of the matrix partial orders that include $\{1\}-$inverses on their definitions. In \cite{Djord} the definition of space pre-order was extended to the class of bounded linear operators on Banach spaces. 
\begin{defn}\label{D: SpacePreOrd}
Let $\f, \g\in \ed_k(\mathcal{B})$ be two bounded linear operators over a Banach space $\mathcal{B}.$ Then $\f$ is said to be below $\g$ under the space pre-order if $\Img(\f) \subseteq \Img(\g)$ and $\Ker(\g) \subseteq \Ker(\f) $ (or equivalently $\Img(\f^*)\subseteq \Img(\g^*)$). We will denote the space pre-order by $<^s$ and we will write $\f <^s \g,$ whenever $\f$ is below $\g$ under the space pre-order. 
\end{defn}

\section{Some remarks on the Moore-Penrose inverse of a bounded finite potent operator}\label{S: MP bounded operator}

The aim of this section is to include some results related to the Moore-Penrose inverse of a bounded finite potent endomorphism that are not present in literature and that will be used in the rest of the paper.

\subsection{On the closedness of the image of a bounded finite potent operator.}\label{ss: ImgBFP}

Given a Hilbert space $\h$ it is well known that the closedness of the image of an operator $\f$ is related to the solvability of the operator equation $\f x=y.$ Moreover, the closedness of $\Img(\f)$ is equivalent to $\f$ being relatively regular, this is, for $\f \in B(\h)$ admitting a bounded $\{1\}-$inverse $\f^-\in B(\h)$ i.e.  $\f \circ \f^- \circ \f = \f,$ so that if $\f x=y$ can be solved then $x=\f^- y$ is a solution. For details, reader is directed to \cite{UnbClos} and \cite{Closed} . There are a lot of important applications of the closedness of the image in perturbation theory and in the context of the spectral study of differential equations, see for instance \cite{Gold}. For an interested reader let just point out that the study of closedness of the image is dealt within Banach spaces too, with several applications, for example, \cite{Atkin}. Now notice that in \cite{MPFP}, the authors pose the study of the Moore-Penrose inverse of a bounded finite potent operator $\f \in B_{fp}(\h)$ by adding the so-called admissible (for the Moore-Penrose) condition on $\f$ (Definition \ref{D: Admissible}), and proving Lemma \ref{L: Adsii ImgCer} for bounded finite potent operators. However, being a bounded finite potent operator could somehow imply any condition on the closedness of the image. This question has not being studied in the framework of bounded finite potent endomorphisms and the author considers that it is important to determine whether the admissibility condition is redundant or necessary. \\We then devote this short section to solve this question. Precisely, a counter example is given, showing that bounded finite potent operators do not have closed image in general, this is, they are not relatively regular operators and therefore, if $\f \in B_{fp}(\h)$ then $\f$ is not necessarily admissible for a Moore-Penrose inverse.
\begin{cex}\label{cex: ImgBFP}
If for every $n\in \N$ we denote by $e_n$ to the sequence: $$e_n=(0, \dots, \stackrel{(n)}{1},0, \dots), $$ then the family (of sequences) $\{e_n \}_{n\in \N}$ is a (Schauder) basis of $\ell^p,$ $1 \leq p < \infty.$ Now, let us consider the following operator: $$\begin{array}{rccl}
\f \colon & \ell^2  & \rightarrow & \ell^2 \\
& e_n & \mapsto & \left \{ \begin{array}{ccl} 0 & \text{ if } & \text{ n is odd } \medskip \\  \dfrac{e_{n+1}}{n+1} & \text{ if } & \text{ n is even } \end{array} \right . ,
\end{array}$$ this is, $\f$ is the composition of the operator consisting on multiplying by the sequence $(0,\frac{1}{3},0, \frac{1}{5},0,\frac{1}{7},0\dots )$ and then the operator of right traslation (or unilateral shift operator), which are both continuous. In fact, $\f$ is continuous and compact, and it is clearly a bounded (continuous) finite potent operator (it is nilpotent).\\ Notice that $\f^2=0$ so $\Img(\f)^2$ is closed. However, $\Img(\f)$ is not closed: the vectors with a finite number of non-null components $(0,0,\frac{1}{3},0,\frac{1}{5},0,\dots , \frac{1}{2n+1},0,0,0, \dots)$ for $n\geq 1,$ are in the image but its limit, in $\ell^2,$ is not. Indeed recall, compact operators do not have closed image unless they have finite rank.
\end{cex} 

Further in this article, we will be interested in finite potent endomorphisms with index lesser or equal than one. So let us make some remarks on them.\\
Despite the previous counterexample, we shall point out that in the case of a finite potent endomorphism of index lesser or equal than one over any inner product vector space, we shall not add any hypothesis in order to obtain the admissibility for the Moore-Penrose inverse. Notice that if $(V,g)$ is any inner product vector space over $k=\R$ or $k=\C,$ and $\f\in \ed_k(V)$ with $i(\f)\leq 1$ then, the AST-decomposition it induces is $V=W_{\f}\oplus U_{\f}=\Img(\f)\oplus\Ker(\f)$ and by definition $W_{\f}=\Img(\f)$ is a finite dimensional k-vector subspace. Therefore, $\Img(\f)$ it is a closed subspace and $V$ also admits the decomposition: $V=\Img(\f)\oplus[\Img(\f)]^{\perp}.$ In short:

\begin{lem}\label{R: FPileq1isAdmissible}
Let $(V,g)$ be any inner product vector space over $k=\R$ or $k=\C$ and let $\f\in \ed_k(V)$ be any finite potent endomorphism with $i(\f)\leq 1,$ then $\Img(\f)$ is a closed subspace of $V.$
\end{lem}

 Let $\h$ be a Hilbert space and let $\f\in \mathcal{B}(\h)$ be a bounded finite potent endomorphism with $i(\f)\leq 1.$ It is already known that $\f^{\dagger}$ is bounded (if $\f$ is so) and that in general, it needs not to be finite potent. However, it is natural to ask if this is changed by $\f$ being of index lesser or equal than one.

\begin{prop}\label{P: MPdeind1}
Given a bounded finite potent endomorphism $\f\in \mathcal{B}(\h)$ with $i(\f)\leq 1$ then $\f^{\dagger}$ is also a bounded finite potent endomorphism with $i(\f^{\dagger})\leq 1.$
\end{prop}
\begin{proof}
This result is a consequence of the relationship between the Moore-Penrose inverse and the adjoint operator. Recall that if $\f$ is a bounded finite potent endomorphism, so is $\f^*$ and moreover $i(\f)=i(\f^*)$ (see Section \ref{ss: AdjoiintBFP}). Therefore $i(\f^*)\leq 1,$ so $V=W_{\f^*}\oplus U_{\f^*}=\Img(\f^*)\oplus\Ker(\f^*).$ The claim is proved as $\Ker(\f^{\dagger})=\Ker(\f^*)$ and $\Img(\f^{\dagger})=\Img(\f^*),$ both being closed $\f^{\dagger}-$invariant subspaces of $\h$ with $V=\Img(\f^{\dagger})\oplus\Ker(\f^{\dagger}).$ In fact, $(\f^{\dagger})_{\vert_{\Img(\f^{\dagger})}}$ is an homeomorphism with $\Img(\f^{\dagger})$ being finite dimensional and $(\f^{\dagger})_{\vert_{\Ker(\f^{\dagger})}}$ a bounded nilpotent operator. 
\end{proof}

\section{Core Inverse of Finite Potent Endomorphisms} \label{s:core-inv-fp}
The study of the Core-Inverse could be approached in two different ways, either as a restriction of the left-Drazin-Moore-Penrose inverse, ``lDMP''; to index $1$ endomorphisms or using Baksalary and Trenkler's original definition (see \cite{BakTren}). Notice that F. Pablos, in \cite{FpDMP} , studied the ``lDMP'' inverse of finite potent endomorphisms and therefore we shall use the other approach to provide a new vision of the theory with new results. This is, we will start by proving the equivalence of Baksalary and Trenkler's definition with the restriction of ``lDMP'' inverses to index $1$ endomorphisms and later we shall study the properties that are not shared between Core-Inverses and ``lDMP'' inverses in general and thus, they are not present already in \cite{FpDMP}. We must emphasize that this new properties are a consequence of the well known equality $\f^{D}=\f^{\#}$ when $i(\f)=1$ (recall the obvious fact that $\f^{D}\in X_{\f}(2)$ in general and $\f^D=\f^{\#}\in X_{\f}(1,2)$ when $i(\f)=1$). All this shall be done within the framework of bounded finite potent endomorphisms (over arbitrary inner product space over a field) that are admissible for the Moore-Penrose inverse.

\medskip
Let $(V,g)$  be an inner product vector space over $k$, with $k = \mathbb C$ or $k = \mathbb R$. In particular, $V$ can be an infinite-dimensional vector space. 

\begin{defn} \label{D:core-inver-fp93735} 

Given a finite potent endomorphism $\varphi \in \ed_k (V)$ admissible for the Moore-Penrose inverse, we say that a linear map $\varphi^{\textcircled\#} \in \ed_k (V)$ is a ``Core Inverse'' of $\varphi$ when:

\begin{itemize}

\item $\varphi \circ \varphi^{\textcircled\#} = P_{\Img(\f)}$;

\item $\Img\varphi^{\textcircled\#} \subseteq \Img\varphi$;
\end{itemize}

where $P_{\Img(\f) } $ is the orthogonal projection induced by the decomposition $$V = \Img(\f) \oplus [\Img(\f)]^\perp\, .$$
\end{defn}

 Using the above mentioned properties of the Moore-Penrose inverse, the two conditions referred to in Definition \ref{D:core-inver-fp93735} can be replaced by the following:

\begin{itemize}
\item $\varphi \circ \varphi^{\textcircled\#} = \varphi \circ \varphi^\dagger$;

\item $P_{\text{Im } \f} \circ \varphi^{\textcircled\#} = \varphi^{\textcircled\#}, \text{ or equivalently } \f\circ \f^{\dagger}\circ \f^{\textcircled\#}=\f^{\textcircled\#}. $
\end{itemize}

Moreover, by substituting the first condition of the definition in the second one gets: \begin{equation}\label{eq: fcircffcirc2}
(\f\circ \f^{\textcircled\#})\circ \f^{\textcircled\#}=\f\circ(\f^{\textcircled\#})^2=\f^{\textcircled\#}
\end{equation}
and by iteration, one reaches that: 
\begin{equation}\label{eq: fcircfnfcircn}
\f^{\textcircled\#}=\f^{n-1}\circ(\f^{\textcircled\#})^n \text{ for every }n>1.
\end{equation}

\begin{prop}\label{P: ConsecuenciasDefinicion}
Given a finite potent endomorphism $\varphi \in \ed_k (V)$ admissible for the Moore-Penrose inverse, if $\f^{\textcircled\#}$ exists, then $i(\f)\leq 1.$
\end{prop}
\begin{proof}
Firstly notice that from $\varphi \circ \varphi^{\textcircled\#} = \varphi \circ \varphi^\dagger$ and $  \f\circ \f^{\dagger}\circ \f^{\textcircled\#}=\f^{\textcircled\#}$ one deduces that $$\f\circ\f^{\textcircled\#}\circ\f=\f, $$ this is, $\f^{\textcircled\#}\in X_{\f}(1).$ Moreover, using \eqref{eq: fcircfnfcircn} and substituting it on $\f\circ\f^{\textcircled\#}\circ\f=\f, $ one gets: $$\f=\f^{n}\circ (\f^{\textcircled\#})^n\circ\f, \text{ for every  }n>1. $$ Therefore, $\Img(\f)\subseteq \Img(\f^n)$ for every $n>1,$ and as the other inclusion is true for every linear operator we conclude that: $$\Img(\f)^{i(\f)}=W_{\f}=\Img(\f) $$ and therefore $i(\f)\leq 1$ (recall \eqref{eq:index1}) as we wanted to prove.
\end{proof}

\begin{cor}\label{C: fcircledfpdeind1}
Given a finite potent endomorphism $\varphi \in \ed_k (V),$ if $\f^{\textcircled\#}$ exists, then $\f^{\textcircled\#}$ is a finite potent endomorphism with $i(\f^{\textcircled\#})\leq 1.$ Moreover, $\f^{\textcircled\#}$ is unique. 
\end{cor}
\begin{proof}
The first statement is a direct consequence of Proposition \ref{P: ConsecuenciasDefinicion} and Definition \ref{D:core-inver-fp93735}, as $W_{\f}=\Img(\f)$ is finite dimensional by definition of AST-decomposition and $\Img(\f^{\textcircled\#})\subseteq \Img(\f).$ Now, the second statement is a consequence of the first statement and the equation showed in \eqref{eq: fcircfnfcircn}. To wit, as $\f^{\textcircled\#}$ is finite potent, let us consider $u\in U_{\f^{\textcircled\#}}.$ Then, if $i(\f^{\textcircled\#})=m,$ it is: $$\f^{\textcircled\#}(u)=\f^{m-1}(\f^{\textcircled\#})^m(u)=0, $$ (recall, by definition of AST-decomposition, $(\f^{\textcircled\#})^m_{_\vert U_{\f^{\textcircled\#}}}=0$). Hence, $$\Ker(\f^{\textcircled\#})=U_{\f^{\textcircled\#}} $$ (as the other inclusion is always true for any finite potent endomorphism) and by \eqref{eq:index1} we conclude that $i(\f^{\textcircled\#})\leq 1.$\\In order to see the uniqueness, let us suppose that there exist some $\f_1^{\textcircled\#}, \f_2^{\textcircled\#}$ satisfying Definition \ref{D:core-inver-fp93735}. Then: $$\f\circ\f_1^{\textcircled\#}=\f\circ\f^{\dagger}=\f\circ\f_2^{\textcircled\#} $$ with $$\Img(\f_1^{\textcircled\#})\subseteq \Img(\f) \text{ and } \Img(\f_2^{\textcircled\#})\subseteq \Img(\f).$$ Hence, \begin{align*}
\Img(\f_1^{\textcircled\#}-\f_2^{\textcircled\#})& \subseteq \Ker(\f); \\
\Img(\f_1^{\textcircled\#}-\f_2^{\textcircled\#})& \subseteq \Img(\f);
\end{align*} 
so it is $$\Img(\f_1^{\textcircled\#}-\f_2^{\textcircled\#}) \subseteq \Ker(\f)\cap\Img(\f)=\{0\};$$ as $i(\f)\leq 1$ so we conclude that $\f_1^{\textcircled\#}=\f_2^{\textcircled\#}.$
\end{proof}

Let us now check the equivalence between Definition \ref{D:core-inver-fp93735} and the restriction to index $1$ of the definition of the left-Drazin-Moore-Penrose inverse (\cite[Theorem 3.2]{FpDMP}). In fact, this result is the generalization to arbitrary vector spaces of $i)$ of \cite[Theorem 1]{BakTren}. This will give a purely algebraic definition of the Core-Inverse of a finite potent endomorphism.

\begin{thm}\textbf{(Algebraic characterization of the Core Inverse.)}\label{T: AlgDefCore-Inver}
Given a finite potent endomorphism $\varphi \in \ed_k (V)$ with $i(\f)\leq 1$, then: $$\f^{\textcircled\#}=\f^{\#}\circ \f \circ \f^{\dagger}.$$ 
\end{thm}
\begin{proof}
Firstly, let us proof that Definition \ref{D:core-inver-fp93735} implies the above expression of $\f^{\textcircled\#}.$ Considering the decomposition induced in $V=\Img(\f)\oplus[\Img(\f)]^{\perp},$ one has that: 
$$(\f^{\#}\circ\f\circ\f^{\dagger})_{\vert_{\Img(\f)}}=(\f^{\#})_{\vert_{\Img(\f)}}=(\f_{\vert_{\Img(\f)}})^{-1}.$$
From the equality, $\f\circ\f^{\textcircled\#}=\f\circ\f^{\dagger},$ one deduces that: $$(\f\circ\f^{\textcircled\#})_{\vert_{\Img(\f)}}=(\f\circ\f^{\dagger})_{\vert_{\Img(\f)}}=Id_{\vert_{\Img(\f)}}, $$ and therefore, as $i(\f)\leq 1$: $$(\f^{\textcircled\#})_{\vert_{\Img(\f)}}=(\f_{\vert_{\Img(\f)}})^{-1}.$$ Hence, $$(\f^{\textcircled\#})_{\vert_{\Img(\f)}}=(\f^{\#}\circ\f\circ\f^{\dagger})_{\vert_{\Img(\f)}}.$$
On the other hand, by the expression of the Moore-Penrose inverse: $$(\f\circ\f^{\textcircled \#})_{\vert_{[\Img(\f)]^{\perp}}}=(\f\circ\f^{\dagger})_{\vert_{[\Img(\f)]^{\perp}}}=0,$$ so $(\f^{\textcircled\#})_{\vert_{[\Img(\f)]^{\perp}}}\in \Ker(\f).$ Therefore, $$(\f^{\textcircled\#})_{\vert_{[\Img(\f)]^{\perp}}}=0 $$ because $\Img(\f^{\textcircled\#})\subseteq \Img(\f)$ and $\Img(\f)\cap\Ker(\f)=\{0\}$ as $i(\f)\leq 1.$ Directly, $$(\f^{\#}\circ\f\circ\f^{\dagger})_{\vert_{[\Img(\f)]^{\perp}}}=0.$$ \medskip Conversely, if $\f^{\textcircled\#}=\f^{\#}\circ\f\circ\f^{\dagger},$ it is straightforward that $$\f\circ\f^{\textcircled\#}=\f\circ(\f^{\#}\circ\f\circ\f^{\dagger})=\f\circ\f^{\dagger}$$ and it is clear that $\Img(\f^{\textcircled\#})=\Img(\f^{\#}\circ\f\circ\f^{\dagger})\subseteq\Img(\f^{\#})=W_{\f}=\Img(\f)$ as $i(\f)\leq 1,$ so the statement is proven.
\end{proof}

\begin{thm}\textbf{(Geometric characterization of the Core Inverse.)}\label{T: GeoCaractCoreInv}
Given a finite potent endomorphism $\varphi \in \ed_k (V)$ with $i(\f)\leq 1$, then $\f^{\textcircled\#}$ is characterized by: $$\f^{\textcircled\#}(v) = \left \{ \begin{array}{ccl} (\f_{\vert_{\Img(\f)}})^{-1}(v) & \text{ if } & v\in \Img(\f) \smallskip\\  0 & \text{ if } & v\in [\Img(\f)]^{\perp} \end{array} \right . ,$$
\end{thm}
\begin{proof}
As in this hypothesis $\Img(\f)$ is closed (Lemma \ref{R: FPileq1isAdmissible}), $V=\Img(\f)\oplus[\Img(\f)]^{\perp}.$  If $v\in \Img(\f),$ then: $$\f^{\textcircled\#}(v)=\f^{\#}\f\f^{\dagger}(v)=\f^{\#}(v)=(\f_{\vert_{\Img(\f)}})^{-1}(v), $$ which makes sense because $\f_{\vert_{\Img(\f)}}=\f_{\vert_{W_{\f}}}$ is an automorphism as $i(\f)\leq 1.$ On the other hand, if $\bar{v}\in [\Img(\f)]^{\perp}$ then: $$\f^{\textcircled\#}(\bar{v})=\f^{\#}\f\f^{\dagger}(\bar{v})=0. $$ Now let us suppose that there exists $\tilde{\f}\in \ed_k(V)$ such that: 
$$\tilde{\f}(v) = \left \{ \begin{array}{ccl} (\f_{\vert_{\Img(\f)}})^{-1}(v) & \text{ if } & v\in \Img(\f) \smallskip\\  0 & \text{ if } & v\in [\Img(\f)]^{\perp} \end{array} \right . .$$ Notice that if $v\in \Img(\f)$ then: $$\f\tilde{\f}(v)=v=\f\f^{\dagger}(v). $$ If $\bar{v}\in [\Img(\f)]^{\perp}$ it is: $$\f\tilde{\f}(\bar{v})=0=\f\f^{\dagger}(\bar{v}). $$ Therefore, $$\f\circ\tilde{\f}=\f\circ\f^{\dagger}.$$ Directly from the expression of $\tilde{\f}$ one deduces that $\Img(\tilde{\f})=\Img(\f)$ and therefore Definition \ref{D:core-inver-fp93735} is satisfied.
\end{proof}

\begin{thm}\label{T: CoreExistsIleq1}
Let us consider a finite potent endomorphism $\f \in \ed_k(V).$ The Core Inverse of $\f$ exists if and only if $i(\f)\leq 1.$
\end{thm}
\begin{proof}
Let us suppose that the Core Inverse of $\f$ exists. Then, it was already proven in Proposition \ref{P: ConsecuenciasDefinicion} that $i(\f)\leq 1.$ Conversely, if $i(\f)\leq 1$ then $V=W_{\f}\oplus U_{\f}=\Img(\f)\oplus \Ker(\f)$ and the $\Img(\f)$ is a finite dimensional vector subspace and hence it is a closed subspace of $V.$ In this conditions $V=\Img(\f)\oplus [\Img(\f)]^{\perp}.$ Therefore, we can now copy the reasoning presented in proving the converse of Theorem \ref{T: GeoCaractCoreInv} for proving that $\f^{\textcircled\#}$ exists.
\end{proof}

\begin{cor}\label{C: MPdela CoreInv}
Given a finite potent endomorphism $\varphi \in \ed_k (V)$ with $i(\f)\leq 1$, then $\f^{\textcircled\#}$ is also admissible for the Moore-Penrose inverse. Moreover, $$(\f^{\textcircled\#})^{\dagger}(v) = \left \{ \begin{array}{ccl} \f(v) & \text{ if } & v\in\Img(\f)\smallskip \\  0 & \text{ if } & v\in [\Img(\f)]^{\perp}\end{array} \right . .$$
\end{cor}
\begin{proof}
As $\f$ is admissible for the Moore-Penrose inverse, then $V=\Img(\f)\oplus[\Img(\f)]^{\perp}.$ Clearly, as $\Img(\f^{\textcircled\#})=\Img(\f^{\#}\circ\f\circ\f^{\dagger})=\Img(\f),$ one deduces that $$V=\Img(\f^{\textcircled\#})\oplus[\Img(\f)^{\textcircled\#}]^{\perp}=\Img(\f)\oplus[\Img(\f)]^{\perp}.$$ For the second statement, bearing in mind that $\Ker(\f^{\textcircled\#})=\Ker(\f^{\#}\circ\f\circ\f^{\dagger})=[\Img(\f)]^{\perp},$ by the characterization of the Moore-Penrose inverse: $$(\f^{\textcircled\#})^{\dagger}(v) = \left \{ \begin{array}{ccl} ((\f^{\textcircled\#})_{\vert_{[\Ker(\f^{\textcircled\#})]^{\perp}}})^{-1}(v) & \text{ if } & v\in\Img(\f^{\textcircled\#}) \\  0 & \text{ if } & v\in [\Img(\f^{\textcircled\#})]^{\perp}\end{array} \right . ;$$ where by the mentioned relations, can be rewritten into:
$$(\f^{\textcircled\#})^{\dagger}(v)= \left \{ \begin{array}{ccl} ((\f^{\textcircled\#})_{\vert_{[\Img(\f^{\textcircled\#})]}})^{-1}(v) & \text{ if } & v\in\Img(\f) \\  0 & \text{ if } & v\in [\Img(\f)]^{\perp}\end{array} \right . ; $$
and finally, as $(\f^{\textcircled\#})_{_{\vert\Img(\f)}}=(\f_{_{\vert\Img(\f)}})^{-1}$ as it was stated in Theorem \ref{T: GeoCaractCoreInv}, one reaches that:

$$(\f^{\textcircled\#})^{\dagger}(v)= \left \{ \begin{array}{ccl} \f(v) & \text{ if } & v\in\Img(\f) \smallskip\\  0 & \text{ if } & v\in [\Img(\f)]^{\perp}\end{array} \right . .$$  
\end{proof}

\begin{cor}\label{C: CoreofCore}
Given a finite potent endomorphism $\varphi \in \ed_k (V)$ with $i(\f)\leq 1$, then $$(\f^{\textcircled\#})^{\textcircled\#}=(\f^{\textcircled\#})^{\dagger}.$$
\end{cor}
\begin{proof}
Firstly we shall point out that this statement has sense in virtue of Corollary \ref{C: fcircledfpdeind1}, this is, $i(\f^{\textcircled\#})\leq 1.$\\
For any $v\in V,$ by Theorem \ref{T: GeoCaractCoreInv}: $$(\f^{\textcircled\#})^{\textcircled\#}(v) = \left \{ \begin{array}{ccl} ((\f^{\textcircled\#})_{\vert_{\Img(\f^{\textcircled\#})}})^{-1}(v) & \text{ if } & v\in\Img(\f^{\textcircled\#})\smallskip \\  0 & \text{ if } & v\in [\Img(\f^{\textcircled\#})]^{\perp}\end{array} \right . ,$$ as $\Img(\f^{\textcircled\#})=\Img(\f)$ and bearing in mind that $$(\f^{\textcircled\#})_{\vert_{\Img(\f)}}=(\f_{\vert_{\Img(\f)}})^{-1}; $$ we can rewrite the above expression as:  $$(\f^{\textcircled\#})^{\textcircled\#}(v) = \left \{ \begin{array}{ccl} ((\f_{\vert_{\Img(\f)}})^{-1})^{-1}(v) & \text{ if } & v\in\Img(\f)\smallskip \\  0 & \text{ if } & v\in [\Img(\f)]^{\perp}  \end{array} \right . .$$ Hence, we conclude by Corollary \ref{C: MPdela CoreInv}.
\end{proof}

\begin{rem}\label{R: fcorecore fcoredagger}
Notice that the previous propositions can be abbreviated into: \begin{equation}\label{eq: MPdelaCreInvAlg}
(\f^{\textcircled\#})^{\textcircled\#}=(\f^{\textcircled\#})^{\dagger}=\f\circ P_{\Img(\f)},
\end{equation}
which is coherent with $iii)$ in \cite[Theorem 1]{BakTren}. Moreover, we shall highlight that $(\f^{\textcircled\#})^{\dagger}$ is a finite potent endomorphism, in general, the Moore-Penrose inverse of a finite potent endomorphism is not a finite potent endomorphism. This is deduced directly from the fact that: $$((\f^{\textcircled\#})^{\dagger})^n=(\f\circ P_{\Img(\f)})^n=\f^n\circ P_{\Img(\f)}.$$
\end{rem}

\begin{cor}\label{L: CoreisEP}
Given a finite potent endomorphism $\varphi \in \ed_k (V)$ with $i(\f)\leq 1$, then $\f^{\textcircled\#}$ is EP. Moreover, $(\f^{\textcircled\#})^{\dagger}$ is also EP.
\end{cor}
\begin{proof}
By the expression obtained in Corollary \ref{C: MPdela CoreInv}, it is straightforward to check that: $$\f^{\textcircled\#}\circ(\f^{\textcircled\#})^{\dagger}=(\f^{\textcircled\#})^{\dagger}\circ\f^{\textcircled\#}.$$ The last statement is deduced from the fact that for any finite potent endomorphism admissible for the Moore-Penrose inverse: $((\f)^{\dagger})^{\dagger}=(\f).$
\end{proof}

\begin{cor}\label{C: GrupoDrazinyMPcoinciden}
Given a finite potent endomorphism $\varphi \in \ed_k (V)$ with $i(\f)\leq 1$, then: $$(\f^{\textcircled\#})^{\dagger}=(\f^{\textcircled\#})^{\#}=(\f^{\textcircled\#})^{D}. $$
\end{cor}
\begin{proof}
This is a direct consequence of \cite[Proposition 3.13]{FpaGroup}, which states that a finite potent endomorphism admissible for Moore-Penrose inverse is EP if and only if $\f^{\#}=\f^{\dagger}.$ The other equality holds due to the well known fact that $\f^{D}=\f^{\#}$ when $i(\f)\leq 1.$
\end{proof}

\begin{cor}\label{C: PropAlge}
Given a finite potent endomorphism $\varphi \in \ed_k (V)$ with $i(\f)\leq 1$, then: \begin{itemize}
\item[•]$\f^{\textcircled\#}\in X_{\f}(1,2)$ \smallskip;
\item[•]$(\f^{\textcircled\#})^2\circ\f=\f^{\#}$ \smallskip;
\item[•]$(\f^{\textcircled\#})^m=(\f^m)^{\textcircled\#}$ \smallskip;
\item[•]$\f^{\textcircled\#} \circ \f=\f^{\#}\circ \f.$ 
\end{itemize}
\end{cor}
\begin{proof}
The first statement is straightforward from the algebraic expression obtained for the Core Inverse in Theorem \ref{T: AlgDefCore-Inver}. The second one can be proven by definitions using the commutativity from the group inverse, \begin{align*}
(\f^{\textcircled\#})^2\circ\f & = \f^{\textcircled\#}\circ\f^{\#}\circ\f\circ\f^{\dagger}\circ\f=\f^{\textcircled\#}\circ\f^{\#}\circ\f=\f^{\textcircled\#}\circ\f\circ\f^{\#}=\\ & =\f^{\#}\circ\f\circ\f^{\dagger}\circ\f\circ\f^{\#}=\f^{\#}\circ\f\circ\f^{\#}=\\&=\f^{\#}.
\end{align*} For the third claim, let $V=W_{\f}\oplus U_{\f}=\Img(\f)\oplus \Ker(\f)$ be the AST decomposition of $V$ induced by $\f$ in our conditions. For all $m\in \Z^+,$ one has that: $W_{\f^m}=W_{\f}$ and $U_{\f^m}=U_{\f}.$ Therefore, the claim is deduced from the fact that: $$([\f^m]_{\vert_{\Img(\f)}})^{-1}=([(\f_{\vert_{\Img(\f)}})]^{-1})^m.$$\\The last equality is straightforward: $\f^{\textcircled\#}\circ\f=\f^{\#}\circ\f\circ\f^{\dagger}\circ\f=\f^{\#}\circ\f.$ 
\end{proof}

\begin{lem}\label{L: fEPsifgroupesfcirc}
Given a finite potent endomorphism $\varphi \in \ed_k (V)$ with $i(\f)\leq 1$, then: $$\f^{\textcircled\#}=\f^{\#} \text{ if and only if } \f \text{ is EP. } $$
\end{lem}
\begin{proof}
If $\f^{\textcircled\#}=\f^{\#}$ then, in particular: $\Ker(\f^{\textcircled\#})=\Ker(\f^{\#}),$ so: $[\Img(\f)]^{\perp}=\Ker(\f)$ and as $\Img(\f)$ is closed in our hypothesis, taking orthogonal yields as desired: $\Img(\f)=[\Ker(\f)]^{\perp}=\Img(\f^*).$\\Conversely, if $\Img(\f)=\Img(\f^*)$ then $\Img(\f)=[\Ker(\f)]^{\perp}$ and therefore: $[\Img(\f)]^{\perp}=\Ker(\f)$ so the claim is deduced by the expressions of both the  Core Inverse and the Group Inverse, recall: $$\f^{\textcircled\#}(v) = \left \{ \begin{array}{ccl} (\f_{\vert_{\Img(\f)}})^{-1} & \text{ if } & v\in \Img(\f)\\ 0 & \text{ if } & v\in [\Img(\f)]^{\perp} \end{array} \right . ,$$ and    $$\f^{\#}(v) = \left \{ \begin{array}{ccl} (\f_{\vert_{\Img(\f)}})^{-1} & \text{ if } & v\in \Img(\f)\\ 0 & \text{ if } & v\in \Ker(\f) \end{array} \right . ,$$ as $i(\f)\leq 1.$
\end{proof}

On a similar way to \cite[Theorem 2]{BakTren} let us highlight some properties of the Core-Inverse.                                                                                                                                                                                                                                                                                                                                                                                                                                                                                                                                                                                                                                                                                                                                                                                                                                                                                                                                                                                                                                                                                                                                                                                                                                                                                                                                                                                                                                                                                                                                                                                                                                                                                                                                                                                                                                                                                                                                                                                                                                                                                                                                                                                                                                                                                                                                                                                                                                                                                                                                                                                                                                                                                                                                                                                                                                                                                                                                                                                                                                                                                                                                                                                                                                                                                                                                                                                                                                                                                                                                                                                                                                                                                                                                                                                                                                                                                                                                                                                                                                                                                                                                                                                                                                                                                                                                                                                                                                                                                                                                                                                                                                                                                                                                                                                                                                                                                                                                                                                                                                                                                                                                                                                                                                                                                                                                                                                                                                                                                                                                                                                                                                                                                                                                                                                                                                                                                                                                                                                                                                                                                                                                                                                                                                                                                                                                                                                                                                                                                                                                                                                                                                                                                                                                                                                                                                                                                                                                                                                                                                                                                                                                                                                                                                                                                                                                                                                                                                                                                                                                                                                                                                                                                                                                                                                                                                                                                                                                                                                                                                                                                                                                                                                                                                                                                                                                                                                                                                                                                                                                                                                                                                                                                                                                                                                                                                                                                                                                                                                                                                                                                                                                                                                                                                                                                                                                                                                                                                                                                                                                                                                                                                                                                                                                                                                                                                                                                                                                                                                                                                                                                                                                                                                                                                                                                                                                                                                                                                                                                                                                                                                                                                                                                                                                                                                                                                                                                                                                                                                                                                                                                                                                                                                                                                                                                                                                                                                                                                                                                                                                                                                                                                                                                                                                                                                                                                                                                                                                                                                                                                                                                                                                                                                                                                                                                                                                                                                                                                                                                                                                                                                                                                                                                                                                                                                                                                                                                                                                                                                                                                                                                                                                                                                                                                                                                                                                                                                                                                                                                                                                                                                                                                                                                                                                                                                                                                                                                                                                                                                                                                                                                                                                                                                                                                                                                                                                                                                                                                                                                                                                                                                                                           
\begin{prop}\label{P: PropertiesofCoreInv}
Given a finite potent endomorphism $\varphi \in \ed_k (V)$ with $i(\f)\leq 1$, then: \begin{itemize}
\item[I.)]$\f^{\textcircled\#}=0$ \text{ if and only if } $\f=0$ ;\smallskip
\item[II.)]$\f^{\textcircled\#}=P_{\Img(\f)}$ \text{ if and only if } $\f^2=\f$ ;\smallskip
\item[III.)]$\f^{\textcircled\#}=\f^{\dagger}$ \text{ if and only if } $\f$ is EP ;\smallskip
\item[IV.)]$\f^{\textcircled\#}=\f$ \text{ if and only if } $\f^3=\f$ and $\f$ is EP  ;\smallskip
\end{itemize}
\end{prop}
\begin{proof}
For $I),$ notice that $\f^{\textcircled\#}=0$ if and only if $\f^{\#}\circ P_{\Img(\f)}=0.$ Pre-composing and post-composing with $\f$ one deduces that $0=\f\circ\f^{\#}\circ\f=\f.$ Conversely, if $\f=0$ then it follows directly from Definition \ref{D:core-inver-fp93735} that $\f^{\textcircled\#}=0.$\\ Statement $II),$ follows from the fact that $\f^2=\f$ if and only if $\f_{\vert_{\Img(\f)}}=Id_{\vert_{\Img(\f)}}.$\\The proof of $III)$ is analogous to the one showed in Lemma \ref{L: fEPsifgroupesfcirc}.\\ Claim number $IV)$ is the restriction to index $1$ of \cite[Proposition 3.14]{FpDMP}.
\end{proof}

\subsection{Some remarks on the Core Inverse for bounded finite potent endomorphisms over Hilbert spaces.}

Let us now include the theorem corresponding to \cite[Theorem 2.1]{Wang} in the framework of our theory. Namely, we will settle $\h$ a Hilbert space in such a way that the adjoint operator of a bounded finite potent endomorphism is well defined.\\Firstly, let us point out that from the algebraic characterization of the Core Inverse deduced in Theorem \ref{T: AlgDefCore-Inver} one obtains immediately that $\f^{\textcircled\#}$ is a continuous (bounded) finite potent operator, as it is composition of three continuous operators.

\begin{thm}\label{T: ExUniCoreInvFP}
If $\f \in \mathcal{B}_{fp}(\h)$ with $i(\f)\leq 1$, then the Core Inverse $\f^{\textcircled\#}$ of $\f$ is the unique linear operator verifying that: \begin{itemize}
\item[I.)]$\f\circ\f^{\textcircled\#}\circ\f=\f;$\smallskip

\item[II.)]$\f\circ (\f^{\textcircled\#})^2=\f^{\textcircled\#};$\smallskip

\item[III.)]$(\f\circ\f^{\textcircled\#})^*=\f\circ\f^{\textcircled\#}.$
\end{itemize}
\end{thm}
\begin{proof}
Let us check that the operator from Definition \ref{D:core-inver-fp93735} satisfies this three conditions. Clearly: $$\f\circ\f^{\textcircled\#}\circ\f=\f\circ\f^{\dagger}\circ\f=\f. $$ The second condition is proven in the same way that the reasoning above \eqref{eq: fcircffcirc2}, substituting in $\f^{\textcircled\#}=\f\circ\f^{\dagger}\circ\f^{\textcircled\#}$ the equality: $\f\circ\f^{\textcircled\#}=\f\circ\f^{\dagger},$ one gets: $$\f\circ (\f^{\textcircled\#})^2=\f^{\textcircled\#}. $$
Finally, the third condition is a direct consequence of the definition of the Moore-Penrose inverse, to wit: $$(\f\circ\f^{\textcircled\#})^*=(\f\circ\f^{\dagger})^*=\f\circ\f^{\dagger}=\f\circ\f^{\textcircled\#}.$$
Finally, if there is any endomorphism $\widehat{\f}$ satisfying the three conditions on the statement, let us check that Definition \ref{D:core-inver-fp93735} holds for him. From condition $II)$ we directly deduce that: $$\Img(\widehat{\f})=\Img(\f\circ(\widehat{\f})^2)\subseteq \Img(\f);$$ From $I)$ we deduce that $$(\f\circ\widehat{\f})_{\vert_{\Img(\f)}}=Id_{\vert_{\Img(\f)}}=(\f\circ\f^{\dagger})_{\vert_{\Img(\f)}}. $$ From $III),$ it is clear that: $$ (\f\circ\widehat{\f})^*=(\widehat{\f})^*\circ \f^*=\f\circ\widehat{\f}.$$ Therefore, $$(\f\circ\widehat{\f})_{\vert_{[\Img(\f)]^{\perp}}}=(\f\circ\widehat{\f})_{\vert_{\Ker(\f^*)}}=((\widehat{\f})^*\circ\f^*)_{\vert_{\Ker(\f^*)}}=0=(\f\circ\f^{\dagger})_{\vert_{[\Img(\f)]^{\perp}}}. $$ Adding all up: $$\f\circ\widehat{\f}=\f\circ\f^{\dagger}, $$ and both conditions of Definition \ref{D:core-inver-fp93735} hold. Once proven the equivalence with Definition \ref{D:core-inver-fp93735} uniqueness is proven in the same way as in Corollary \ref{C: fcircledfpdeind1}.
\end{proof}

Let us calculate the Core Inverse of the Moore-Penrose inverse of a bounded finite potent endomorphism of index lesser of equal than one.

\begin{prop}\label{P: CoreoftheMP}
If $\f\in \mathcal{B}_{fp}(\h)$ with $i(\f)\leq 1,$ then: $$(\f^{\dagger})^{\textcircled\#}=(\f^{\dagger})^{\#}\circ P_{\Img(\f^{\dagger})}. $$
\end{prop}
\begin{proof}
Firstly notice that this statement makes sense due to Proposition \ref{P: MPdeind1}. By direct computation:
$$(\f^{\dagger})^{\textcircled\#}(v)= \left \{ \begin{array}{ccl} ((\f^{\dagger})_{\vert_{\Img(\f^{\dagger})}})^{-1}(v) & \text{ if } & v\in\Img(\f^{\dagger}) \smallskip\\  0 & \text{ if } & v\in [\Img(\f^{\dagger})]^{\perp}\end{array} \right . ,$$ and bearing in mind the characterization of the Moore-Penrose inverse and that $[\Img(\f^{\dagger})]^{\perp}=[\Img(\f^*)]^{\perp}=\Ker(\f),$ this can be expressed as: $$(\f^{\dagger})^{\textcircled\#}(v)= \left \{ \begin{array}{ccl} ((\f)_{\vert_{[\Ker(\f)]^{\perp}}})(v) & \text{ if } & v\in\Img(\f^{\dagger}) \smallskip\\  0 & \text{ if } & v\in \Ker(\f) \end{array} \right . ,$$ from where the claim is deduced.
\end{proof}

\begin{prop}
Given $\f \in \mathcal{B}_{fp}(\h)$ with $i(\f)\leq 1,$ then: $$\f^{\textcircled\#}=\f^*  \text{ if and only if }  \f\circ\f^*\circ\f=\f \text{ and } \f \text{ is EP }.$$
\end{prop}

\begin{proof}
If $\f^{\textcircled\#}=\f^*$ then as $\f\circ\f^{\textcircled\#}\circ\f=\f$ the first claim is clear. Moreover, $\Img(\f^{\textcircled\#})=\Img(\f)=\Img(\f^*).$ Conversely, as $\f\circ\f^*\circ\f=\f$ and $i(\f)\leq 1,$ one has that: $$(\f^*)_{\vert_{\Img(\f)}}=(\f_{\vert_{\Img(\f)}})^{-1}=(\f^{\textcircled\#})_{\vert_{\Img(\f)}}, $$ by Theorem \ref{T: GeoCaractCoreInv}. As $\Ker(\f^*)=[\Img(\f)]^{\perp}$ we conclude that: $$ (\f^*)_{\vert_{[\Img(\f)]^{\perp}}}=0=(\f^{\textcircled\#})_{\vert_{[\Img(\f)]^{\perp}}}. $$ Therefore, we obtain that $\f^*=\f^{\textcircled\#}$ as desired.
\end{proof}

To conclude this section, let us generalize \cite[Theorem 3]{BakTren} to our case.
\begin{thm}\label{T: fEpyequiv}
If $\f\in \mathcal{B}_{p}(\h)$ with $i(\f)\leq 1,$ then the following are equivalent: \begin{itemize}
\item[I.)]$\f$ is EP \smallskip ;
\item[II.)]$(\f^{\textcircled\#})^{\textcircled\#}=\f$\smallskip ;
\item[III.)]$\f^{\textcircled\#}\circ\f=\f\circ\f^{\textcircled\#}$\smallskip ;
\item[IV.)]$(\f^{\dagger})^{\textcircled\#}=\f$ \smallskip ;
\item[V.)]$(\f^{\textcircled\#})^{\dagger}=(\f^{\dagger})^{\dagger}.$
\end{itemize}
\end{thm}
\begin{proof}
Firstly, let us check that $I)$ implies $II).$ As $\f$ is EP, then $\f\circ\f^{\dagger}=\f^{\dagger}\circ\f.$ Therefore, $$\f\circ P_{\Img(\f)}=\f\circ (\f\circ\f^{\dagger})=\f\circ(\f^{\dagger}\circ\f)=\f$$ and we conclude by Corollary \ref{C: CoreofCore}.\\ Conversely, let us see that $II)$ implies $I).$ By the expression obtained in Corollary \ref{C: CoreofCore}, it is clear that $(\f^{\textcircled\#})^{\textcircled\#}=\f$ if and only if $[\Img(\f)]^{\perp}=\Ker(\f),$ this is, if $\Ker(\f^*)=\Ker(\f),$ which is equivalent to $\f$ being EP.\\
That $I)$ occurs if and only if $III)$ occurs was proved in Lemma \ref{L: fEPsifgroupesfcirc}.\\
To see that $I)$ happens if and only if $IV)$ recall that $\Img(\f^{\dagger})=\Img(\f^*).$ By the expression obtained in Proposition \ref{P: CoreoftheMP}, one deduces the equivalence occurs if and only if $\Img(\f^*)=\Img(\f).$\\
Finally, on a similar fashion to previous equivalences, $IV)$ if and only if $V)$ is deduced bearing in mind the expressions obtained both in Corollary \ref{C: MPdela CoreInv} and Proposition \ref{P: CoreoftheMP}. The equivalence happens if and only if $[\Img(\f)]^{\perp}=\Ker(\f^*),$ this is $\Ker(\f)=\Ker(\f^*).$
\end{proof}

\bigskip

\section{Core-Partial order for finite potent endomorphisms}\label{s: Core Partial Order}

Let $(V,g)$ be an inner product arbitrary vector space over $k=\R$ or $k=\C.$ Henceforth, we will denote by $\ed_k^{fp}(V)^{\leq 1}$ the set of finite potent endomorphisms on $V$ of index lesser or equal than $1.$

\begin{defn}\label{D: CorePartialOrder}
Let $\f,\g \in \ed_k^{fp}(V)^{\leq 1}.$ We will say that $\f$ is under $\g$ for the core order, and it will be denoted by $\f \leq^{\textcircled\#}\g ,$ when: \begin{align*}
\f\circ\f^{\textcircled\#}& =\g\circ\f^{\textcircled\#}\smallskip \\
\f^{\textcircled\#}\circ\f &=\f^{\textcircled\#}\circ\g.
\end{align*}
\end{defn}

\begin{lem}\label{L: CoreySpace}
Let $\f,\g \in \ed_k^{fp}(V)^{\leq 1}.$ If $\f {\leq}^{\textcircled\#}\g$ then: \begin{itemize}
\item[•]$\Img(\f)\subseteq \Img(\g);$ \smallskip 
\item[•]$\Ker(\g)\subseteq \Ker(\f).$
\end{itemize}
\end{lem}
\begin{proof}
Recall that for any $\f^{-}\in X_{\f}(1),$ one has that  $\Img(\f)=\Img(\f\circ\f^-)$ and $\Ker(\f)=\Ker(\f^-\circ\f).$
Then, the claims are deduced from Corollary \ref{C: PropAlge} and Definition \ref{D: CorePartialOrder} as: \begin{align*}
\Img(\f)& = \Img(\f\circ\f^{\textcircled\#})=\Img(\g\circ\f^{\textcircled\#})\subseteq \Img(\g); \\
\Ker(\g)& \subseteq \Ker(\f^{\textcircled\#}\circ\g)=\Ker(\f^{\textcircled\#}\circ\f)= \Ker(\f).
\end{align*}
\end{proof}

\begin{cor}
Let $\f,\g \in \ed_k^{fp}(V)^{\leq 1}.$ If $\f {\leq}^{\textcircled\#}\g,$ then $\f<^s \g,$ where $<^s$ denotes the space pre-order (Section \ref{ss: SpacePreord}).
\end{cor}

\begin{thm}\label{T: CaractCore}\textbf{(Characterization of the core order)}.
Let $\f,\g \in \ed_k^{fp}(V)^{\leq 1}.$ If $V=W_{\f}\oplus U_{\f}=\Img(\f)\oplus \Ker(\f)$ is the AST-decomposition of $V,$ then: $$\f \leq^{\textcircled\#}\g \text{ if and only if }\f_{\vert_{\Img(\f)}}=\g_{\vert_{\Img(\f)}} \text{ and } \g_{\vert_{\Ker(\f)}}\subseteq [\Img(\f)]^{\perp}. $$
\end{thm}
\begin{proof}
Notice that if $$\f\circ\f^{\textcircled\#}=\g\circ\f^{\textcircled\#} \text{ then } \f_{\vert_{\Img(\f)}}=\g_{\vert_{\Img(\f)}}.$$ On the other hand, as $(\f^{\textcircled\#}\circ\f)_{\vert_{\Ker(\f)}}=0=(\f^{\textcircled\#}\circ\g)_{\vert_{\Ker(\f)}}$ then $$\g_{\vert_{\Ker(\f)}}\subseteq \Ker(\f^{\textcircled\#})=[\Img(\f)]^{\perp}.$$ Conversely, let us suppose that both conditions on the statement hold and let us prove that the definition of the core order is satisfied. On one side, $$(\f\circ\f^{\textcircled\#})_{\vert_{[\Img(\f)]^{\perp}}}=0=(\g\circ\f^{\textcircled\#})_{\vert_{[\Img(\f)]^{\perp}}}$$ as $\Ker(\f^{\textcircled\#})=[\Img(\f)]^{\perp}.$ Now, as $\f_{\vert_{\Img(\f)}}\in \aut_{k}(\Img(\f)),$ then $$\g\circ(\f_{\vert_{\Img(\f)}})^{-1}=\f\circ(\f_{\vert_{\Img(\f)}})^{-1}. $$ Therefore,  $$(\f\circ\f^{\textcircled\#})_{\vert_{\Img(\f)}}=0=(\g\circ\f^{\textcircled\#})_{\vert_{\Img(\f)}}$$ and it is $$\f\circ\f^{\textcircled\#}=\g\circ\f^{\textcircled\#}. $$
 Moreover, now, $(\f^{\textcircled\#}\circ\f)_{\vert_{\Img(\f)}}=(\f^{\textcircled\#}\circ\g)_{\vert_{\Img(\f)}}.$ As $\g_{\vert_{\Ker(\f)}}\subseteq [\Img(\f)]^{\perp}=\Ker(\f^{\textcircled\#})$ then $$(\f^{\textcircled\#}\circ\g)_{\vert_{\Ker(\f)}}=0=(\f^{\textcircled\#}\circ\f)_{\vert_{\Ker(\f)}} $$ and therefore, $\f^{\textcircled\#}\circ\f=\f^{\textcircled\#}\circ\g$ as we wanted to prove. 
\end{proof}

\begin{cor}\label{C: GdejaInvLaAST}
Let $\f,\g \in \ed_k^{fp}(V)^{\leq 1}$ such that $\f\leq^{\textcircled\#}\g.$ If $\f$ is EP, then $\g$ leaves the AST-decomposition of $\f$ invariant. 
\end{cor}
\begin{proof}
It follows from Theorem \ref{T: CaractCore} and the fact that $[\Img(\f)]^{\perp}=\Ker(\f)$ if $\f$ is EP.
\end{proof}

\begin{lem}\label{L: DescompKer}
Let $\f,\g \in \ed_k^{fp}(V)^{\leq 1}.$ If $\f\leq^{\textcircled\#}\g$ then: $$\Ker(\f)=\Ker(\g)\oplus(\Img(\g)\cap\Ker(\f)).$$
\end{lem}
\begin{proof}
It is an immediate consequence of $\f$ and $\g$ being of index lesser or equal than $1$ and that $\Ker(\g)\subseteq \Ker(\f)$ (recall Lemma \ref{L: CoreySpace}).
\end{proof}

\begin{lem}\label{L: InvImg}
Let $\f,\g,\phi \in \ed_k^{fp}(V)^{\leq 1}.$ If $\f\leq^{\textcircled\#}\g$ and $\g\leq^{\textcircled\#}\phi,$ then $$\f_{\vert_{\Img(\f)}}=\phi_{\vert_{\Img(\f)}}.$$
\end{lem}
\begin{proof}
By Lemma \ref{L: CoreySpace} we know that $\Img(\f)\subseteq \Img(\g).$ By Theorem \ref{T: CaractCore}, as $\g \leq^{\textcircled\#}\g,$ $\phi_{\vert_{\Img(\g)}}=\g_{\vert_{\Img(\g)}}.$ Therefore, as $\f\leq^{\textcircled\#}\g,$ $$\phi_{\vert_{\Img(\f)}}=\g_{\vert_{\Img(\f)}}=\f_{\vert_{\Img(\f)}}$$ and we conclude.
\end{proof}

\begin{thm}\label{T: CoreOrdenParcial}
The core order is a partial order in the set $\ed_k^{fp}(V)^{\leq 1}.$
\end{thm}
\begin{proof}
Reflexivity holds directly. In order to prove anti-symmetry, let us consider $\f,\g\in \ed_k^{fp}(V)^{\leq 1}$ such that $\f\leq^{\textcircled\#}\g$ and $\g\leq^{\textcircled\#}\f.$ By Lemma \ref{L: CoreySpace} we know that $\Img(\f)=\Img(\g)$ and $\Ker(\f)=\Ker(\g).$ As both are finite potent endomorphisms of index lesser or equal than $1$ we conclude that $\f=\g.$ \\For transitivity, let us consider $\f,\g,\phi\in \ed_k^{fp}(V)^{\leq 1}$ such that $\f\leq^{\textcircled\#}\g$ and $\g\leq^{\textcircled\#}\phi.$ Then, the same reasoning showed in the proof of the converse of Theorem \ref{T: CaractCore} enables to extend the equality of Lemma \ref{L: InvImg} to the equality: $$\f\circ \f^{\textcircled\#}=\phi\circ\f^{\textcircled\#}.$$ Moreover, again by Lemma \ref{L: InvImg}: $$(\f^{\textcircled\#}\circ\f)_{\vert_{\Img(\f)}}=(\f^{\textcircled\#}\circ\phi)_{\vert_{\Img(\f)}}. $$ Further, by Lemma \ref{L: DescompKer}: $$(\f^{\textcircled\#}\circ\phi)_{\vert_{\Ker(\f)}}=(\f^{\textcircled\#}\circ\phi)_{\vert_{\Ker(\g)\oplus(\Img(\g)\cap\Ker(\f))}}. $$  If $v\in (\Img(\g)\cap\Ker(\f))$ with $v=\g(v'),$ $$\phi(v)=\phi(\g(v'))=\g(\g(v'))=\g(v), $$ so $\f^{\textcircled\#}\phi(v)=\f^{\textcircled\#}\g(v)$ with $v\in \Ker(\f),$ so by Theorem \ref{T: CaractCore}, $$\f^{\textcircled\#}\g(v)=0=\f^{\textcircled\#}\phi(v) $$ and as $(\f^{\textcircled\#}\phi)_{\vert_{\Ker(\g)}}=0$ we deduce that: $(\f^{\textcircled\#}\circ \phi)_{\vert_{\Ker(\f)}}=0=(\f^{\textcircled\#}\circ\f)_{\vert_{\Ker(\f)}}.$ Finally, it is $$\f^{\textcircled\#}\circ\f=\f^{\textcircled\#}\circ\phi $$ and we conclude.
\end{proof}

Following the discussion presented in \cite[Section 3]{BakTren}, we can prove another characterization of the core partial order.

\begin{thm}\label{T: CaractAlgeCPO}
Let $\f,\g \in \ed_k^{fp}(V)^{\leq 1}.$ Then: \begin{itemize}
\item[•]$\f\circ\f^{\textcircled\#}=\g\circ\f^{\textcircled\#}$ if and only if $\f^2=\g\circ\f ;$ \smallskip
\item[•]$\f^{\textcircled\#}\circ\f=\f^{\textcircled\#}\circ\g$ if and only if $\f^{\dagger}\circ\f=\f^{\dagger}\circ\g.$
\end{itemize}
\end{thm}
\begin{proof}
Let us begin with the first equivalence. Let us suppose that $\f\circ\f^{\textcircled\#}=\g\circ\f^{\textcircled\#}.$ Then, $\f\circ\f^{\dagger}=\g\circ\f^{\#}\circ\f\circ\f^{\dagger}$ by Theorem \ref{T: AlgDefCore-Inver}. Post-composing with $\f^2$ yields: $\f^2=\g\circ\f^{\#}\circ\f^2, $ and by the commuting property of the group inverse it is: $$\f^2=\g\circ\f\circ\f^{\#}\circ\f=\g\circ\f.$$ Conversely, let us suppose that $\f^2=\g\circ\f.$ Post composing with $\f^{\#}\circ\f^{\dagger}$ we get: $\f^2\circ\f^{\#}\circ\f^{\dagger}=\g\circ\f\circ\f^{\#}\circ\f^{\dagger}.$ Again, by the commuting property of the group inverse: $$\f\circ\f^{\#}\circ\f\circ\f^{\dagger}=\g\circ\f^{\#}\circ\f\circ\f^{\dagger}, $$ and we conclude by Theorem \ref{T: AlgDefCore-Inver}.\medskip \\Let us prove the second statement. Let us suppose that $\f^{\textcircled\#}\circ\f=\f^{\textcircled\#}\circ\g,$ which can be written as: $\f^{\#}\circ\f=\f^{\#}\circ \f\circ\f^{\dagger}\circ\g$ in virtue of Theorem \ref{T: AlgDefCore-Inver}. Pre-composing with $\f^{\dagger}\circ\f$ yields: $\f^{\dagger}\circ\f=\f^{\dagger}\circ\f\circ\f^{\#}\circ\f\circ\f^{\dagger}\circ\g,$ using the definition of the group inverse, one gets: $$\f^{\dagger}\circ\f=\f^{\dagger}\circ\f\circ\f^{\dagger}\circ\g=\f^{\dagger}\circ\g.$$ Conversely, if $\f^{\dagger}\circ\f=\f^{\dagger}\circ\g,$ pre-composing with $\f^{\#}\circ\f$ we arrive at $\f^{\textcircled\#}\circ\f=\f^{\textcircled\#}\circ\g$ by Theorem \ref{T: AlgDefCore-Inver}.  
\end{proof}

\section{A pre-order induced by the Core Inverse for finite potent endomorphisms.}\label{S: PreOrder}
Finally, in this short section, let us include an order which makes sense in the case of finite potent endomorphisms as it uses strongly the notion of index and the CN-decomposition of a finite potent endomorphism (Section \ref{ss: CNdecomp01010101}). Moreover, we settle some problems for future research.

\begin{defn}\label{D: GeneralCoreOrder}
Given $\f,\g \in \ed_k^{fp}(V)$ with $\f=\f_1+\f_2$ and $\g=\g_1+\g_2$ their respective CN-decompositions. We will say that $\f$ is below $\g$ for the ``general core order'', and it will be denoted as $\f <^{\textcircled\#}\g,$ when $$\f_1\leq^{\textcircled\#}\g_1 $$ for the core partial order (Definition \ref{D: CorePartialOrder}).
\end{defn}

Let us point out that this definition makes sense as by definition of the CN-decomposition the core part of any endomorphism is of index lesser or equal than one.

\begin{rem}\label{R: MorfismoDelosFP}
Let us consider any finite potent endomorphism $\f\in \ed_k(V)$ with $\f=\f_1+\f_2$ being its CN-decomposition. Let us denote, again, by $\ed_k^{fp}(V)^{\leq 1}$ the set of finite potent endomorphisms of index lesser or equal than one. There always exists a surjective morphism: $$\begin{array}{rccl}
\Gamma \colon & \ed_k^{fp}(V)  & \rightarrow & \ed_k^{fp}(V)^{\leq 1} \\
& \f & \mapsto & \f_1
\end{array} .$$ It is clear that this morphism is an isomorphism when restricting to the set of finite potent endomorphisms of index lesser or equal than one, this is, $\Gamma_{\vert End_k^{fp}(V)^{\leq 1}}.$ \\Notice that we can formulate and answer questions in the theory of matrix partial orders using this morphism. This morphism being an isomorphism when $i(\f)\leq 1$ means that the general core order and the core partial order coincide in the set of finite potent endomorphisms of index lesser or equal than one.\\
On the other hand, the lack of injectivity of this morphism, when applied to the general core order (Definition \ref{D: GeneralCoreOrder}), is the same as saying that this relation is not anti-symmetric. Thus, bearing in mind the well known relationship between finite matrices and endomorphisms over finite dimensional vector spaces, any counterexample for the injectivity of this morphism is a counterexample to the anti-symmetric property of the general core order. For instance, let us consider the following matrices:
$$A=\begin{pmatrix} 29 & 0 & 0 & 0 & 0 \\ 0 & 33 & 0 & 0 & 0 \\ 0 & 0 & 0 & 0 & 0 \\ 0 & 0 & 0 & 0 & 0 \\ 0 & 0 & 0 & 1 & 0 \end{pmatrix} \text{ and } B=\begin{pmatrix} 29 & 0 & 0 & 0 & 0 \\ 0 & 33 & 0 & 0 & 0 \\ 0 & 0 & 0 & 0 & 0 \\ 0 & 0 & 1 & 0 & 0 \\ 0 & 0 & 0 & 1 & 0 \end{pmatrix}\in \text{Mat}_{5\times 5} ({\mathbb R})\, .$$ 
Notice that $i(A)=2$ and $i(B)=3.$ It is evident that $A_1=B_1= \begin{pmatrix}
29 & 0\\
0  & 33 \end{pmatrix}, $ $A_2=\begin{pmatrix}
0 & 0 & 0\\
0 & 0 & 0\\
0 & 1 & 0 
\end{pmatrix}$ and $B_2=\begin{pmatrix}
0 & 0 & 0\\
1 & 0 & 0\\
0 & 1 & 0
\end{pmatrix},$ hence $A<^{\textcircled \#}B$ and $B<^{\textcircled \#}A,$ but however $A\neq B.$

\end{rem}

\begin{thm}\label{T: GeneralCoreOrderIsPreorder}
The general core order is a pre-order in the set of $\ed_k^{fp}(V).$
\end{thm}
\begin{proof}
Reflexivity holds directly and transitivity is guaranteed by Theorem \ref{T: CoreOrdenParcial}.
\end{proof}

\begin{rem}
As it has been shown in Section \ref{s:core-inv-fp}, the Core Inverse makes sense in the framework of finite potent endomorphisms of index lesser or equal than one. This limitation, when working in the framework of finite matrices, led Manjunatha Prasad and Mohana to present the ``Core-EP'' inverse in 2014 as an extension of the Core Inverse for the case of matrices of arbitrary index. This work was done in \cite{Manj}. Moreover, in 2016, H.Wang introduced a new decomposition for square matrices called the Core-EP decomposition and showed some of its applications in \cite{HWang}. Among them, he defined some new orders such as the core-EP order and the core-minus order. The author of the present work thinks that it shall be of mathematical interest to generalize the theory of the Core-EP inverse (with the Core-EP decomposition) to finite potent endomorphisms. Once this is done, it shall be applied to studying the Core-EP order in the set of finite potent endomorphisms of arbitrary index. Finally, the  author thinks that the relation between the ``general core order'' and the Core-EP order shall be clarified in the framework of finite potent endomorphisms and thus, by specialization, in the case of finite square matrices. 

\end{rem}

\section*{Declarations}

\begin{itemize}
\item Funding: no funding was received to assist with the preparation of this manuscript. 
\item The author has no relevant financial or non-financial interests to disclose.
\item Conflict of interest/Competing interests: none.
\item Ethics approval: not-applicable.
\item Consent to participate: not-applicable.
\item Consent for publication: the author authorises the publication of the previous notes.
\item Availability of data and materials: Not-applicable.
\item Code availability: not-applicable.
\item Authors' contributions: Not-applicable.
\end{itemize}

\medskip

\section*{Acknowledgements}
The author would like to thank Dr. Ángel A. Tocino García, from the Department of Mathematics of the University of Salamanca, for his useful help in working out the Counterexample \ref{cex: ImgBFP} offered in Section \ref{ss: ImgBFP}.


\begin{thebibliography}{MMMM}



\bibitem{AST}  Argerami, M., Szechtman, F., Tifenbach, R.; \textit{On Tate's trace}, Linear Multilinear Algebra 55:6,515-520,(2007).

\bibitem{Atkin} Atkinson F.V.; \textit{On relatively regular operators}, Acta Sci. Math. (Szeged) 15:1-1, 38-56,(1954).


\bibitem{BakTren} Baksalary, O. M., Trenkler, G.  \textit{Core inverse of matrices}. Linear Multilinear Algebra, 58(6), 681–697, (2010). 

\bibitem{Ben} Ben-Israel, A., Greville T.N.; \textit{Generalized inverses: theory and applications}, Vol. 15, Springer Science and Business Media, (2003).

\bibitem{MPFP} Cabezas Sánchez,V., Pablos Romo, F.; \textit{Moore-Penrose Inverse os Some Linear Maps on Infinite-Dimensional Vector Spaces}, Electron. J.Linear Algebra 36, 570-586, (2020).

\bibitem{Camp} Campbell, S.L., Meyer, C.D.; \textit{Generalized inverses of linear transformations}, New York, Wiley-Interscience, (1974).

\bibitem{Di1} Mosic D.; \textit{Condition Numbers Related to the Core Inverse of a Complex Matrix}, Filomat, 36(11), 3785-3796, (2022).


\bibitem{Di2} Mosic D.; \textit{One-sided core partial orders on a ring with involution}, Revista de la Real Academia de Ciencias Exactas, Físicas y Naturales, Seria A. Matemáticas, 112(4), 1367-1379, (2018).



\bibitem{Gold} Goldberg S.; \textit{Unbounded Linear Operators}, Mc Graw Hill, New York, (1966).


\bibitem{Huan} Huan Q., Chen S., Guo Z., Zhu L.; \textit{Regular factorizations and perturbation analysis for the core inverse of linear operators in Hilbert spaces}, International Journal of Computer Mathematics, 96(10), 1943-1956, (2019).

\bibitem{Ke} Ke Y., Wang L., Chen J.; \textit{The core inverse of a product and 2x2 matrices}, Bull. Malays. Math. Sci. Soc. 42, 51-66, (2019).



\bibitem{UnbClos} Kulkarni, S.H., Nair M.T., Ramesh, G.; \textit{Some Properties of Unbounded Operators with Closed Range}, Proc. Indian Acad. Sci. (Math.Sci.), Vol.4, 613-625, (2008).


\bibitem{Kurata} Kurata H.; \textit{Some theorems on the core inverse of matrices and the core partial ordening}, Appl. Math. Comput. 316, 43-51, (2018).

\bibitem{Malik} Malik S.B.; \textit{Some more properties of the core partial order}, Appl. Math. Comput. 221, 192-201, (2013).

\bibitem{Manj} Manjunatha Prasad, K., Mohana, K.S.;  \textit{Core-EP inverse}, Linear Multilinear Algebra, 62(3), 792-802, (2014).

\bibitem{Ind}  Mitra, S.K., Bhimasankaram, P.,Malik, S.B.;\textit{Matrix Partial Orders, Shorted Operators and Applications}, World Scientific, (2010).

\bibitem{Djord} Rakić, D.S., Djordjević, D.S. \textit{Space pre-order and minus partial order for operators on Banach spaces}, Aequat. Math. 85, 429–448 (2013). https://doi.org/10.1007/s00010-012-0133-2

\bibitem{Closed} Ould-Ali, M., Messirdi, B.; \textit{On Closed Range Operators in Hilbert Space}, Int. J. Algebra, Vol.4, no.20, 953-958, (2010).


\bibitem{Pa} Pablos Romo, F. \textit{Classification of finite potent endomorphisms}, Linear Algebra Appl., 440, 266-277,(2014).

\bibitem{Pa-CN} Pablos Romo, F.; \textit{Core-Nilpotent Decomposition and new generalized inverses  of  Finite Potent Endomorphisms}, Linear Multilinear Algebra 68:11, 2254-2275, (2020).

\bibitem{Fpa-CN}  Pablos Romo, F.; \textit{Core-Nilpotent Decomposition of Infinite Dimensional Vector Spaces}, Mediterr. J. Math.,(2021).


\bibitem{Fpa-BounGen} Pablos Romo, F.; \textit{Generalized Inverses of Bounded Finite Potent Operators on Hilbert spaces}, Filomat 36:18, 6139-6158, (2022).

\bibitem{FpaGroup} Pablos Romo, F.; \textit{Group inverse of finite potent endomorphisms on arbitrary vector spaces}, Oper. Matrices,  14(4), 1029–1042, (2020).


\bibitem{Fpa-Boun} Pablos Romo, F.; \textit{On bounded Finite Potent Operators on arbitrary Hilbert Spaces}, Bull. Malays. Math. Sci. Soc. 44(6), 4085-4107 (2021).


\bibitem{FpDMP} Pablos Romo, F.; \textit{On Drazin-Moore-Penrose inverses of finite potent endomorphisms}, Linear Multilinear Algebra, 69(4), 627–647. (2019) https://doi.org/10.1080/03081087.2019.1612834


\bibitem{Ta}  Tate, J. \textit{Residues of Differentials on Curves}, Ann. Scient. \'Ec.
Norm. Sup. 1,4a s\'erie, 149-159,(1968).

\bibitem{HWang} Wang, H.; \textit{Core-EP decomposition and its applications}, Linear Algebra Appl. 508, 289-300, (2016).

\bibitem{Wang} Wang, H., Liu, X.; \textit{Characterizations of the core inverse and the core partial ordering} Linear Multilinear Algebra, 63(9), 1829-1836,(2015).


\end{thebibliography}
\end{document}